\documentclass[11pt]{amsart}
\usepackage{amsmath, amssymb, amsfonts, amsthm, mathtools, verbatim, dsfont, graphicx, enumerate, pifont,tikz,tikz-cd}

\usepackage{hyperref}

\usepackage[backend=bibtex]{biblatex}

\theoremstyle{theorem}
\newtheorem{theorem}{Theorem}[section]
\newtheorem{lemma}[theorem]{Lemma}
\newtheorem{proposition}[theorem]{Proposition}

\theoremstyle{definition}
\newtheorem{definition}[theorem]{Definition}
\newtheorem{remark}[theorem]{Remark}

\newtheorem{example}[theorem]{Example}

\providecommand{\abs}[1]{\left\lvert#1\right\rvert}

\providecommand{\OP}[1]{\operatorname{#1}}

\DeclarePairedDelimiter\floor{\lfloor}{\rfloor}

\begin{document}

\author{Igor Uljarevic}
\title{Partitions of the set of Natural numbers and symplectic homology}
\maketitle
\begin{center}\today \end{center}

\begin{abstract}
We prove Tamura's theorem on partitions of the set of positive integers (a generalization of the more famous Rayleigh-Beatty theorem) using the positive $\mathbb{S}^1$-equivariant symplectic homology. 
\end{abstract}

\section{Introduction}

In his famous book, \textit{The Theory of Sound}, John William Strutt noticed the following property of the sequence
\[ n\mapsto \floor{n\cdot \alpha}, \]
where $\floor{\cdot}$ stands for the floor function, and $\alpha$ is a positive irrational number \cite[page~123]{MR0016009}.

\begin{theorem}[Rayleigh]
Let $\alpha$ and $\beta$ be positive irrational numbers such that
\[\frac{1}{\alpha}+\frac{1}{\beta}=1.\]
Then, the sets $\{\floor{n\alpha}\:|\: n\in\mathbb{N}\}$ and $\{\floor{n\beta}\:|\:n\in \mathbb{N}\}$ partition the set of natural numbers $\mathbb{N}.$ 
\end{theorem}

Here and in the rest of the paper, the verb ``partition'' is understood in the following sense.
 
\begin{definition}
Let $m\in \mathbb{N}$ and let  $A,A_1,\ldots, A_m$ be non-empty sets. The sets $A_1,\ldots, A_m$ are said to \textbf{partition} the set $A$ if they are pairwise disjoint (i.e. $A_i\cap A_j=\emptyset$ for $i\not=j$) and satisfy
\[A_1\cup\cdots\cup A_m=A.\]
\end{definition}

The direct generalization of Rayleigh's theorem (partitions of $\mathbb{N}$ in more than 2 subsets) is false, as it was shown by Uspensky \cite{MR1521314}.

\begin{theorem}[Uspensky]
Let $m\in\mathbb{N}$ and let $a_1,\ldots, a_n\in (0,\infty)$ be positive real numbers. If the sets
\[\{\floor{na_i}\:|\:n\in\mathbb{N} \},\quad i=1,\ldots, m,\]
partition $\mathbb{N},$ then either $m=1$ and $a_1=1$ or $m=2$ and $a_1,a_2$ are irrational numbers satisfying $\frac{1}{a_1}+\frac{1}{a_2}=1.$
\end{theorem}

The aim of this paper is to give an alternative proof to the following generalization of Rayleigh's theorem.

\begin{theorem}[Tamura~\cite{MR1374387}]\label{thm:main}
Let $m\in\mathbb{N}$ and let $a_1,\ldots,a_m\in(0,\infty)$ be positive real numbers such that $\frac{a_j}{a_k}\not\in\mathbb{Q}$ for $j\not=k.$ Then, the sets
\[A_j:=\left\{\sum_{k=1}^m\floor*{n\frac{a_j}{a_k}}\:|\: n\in\mathbb{N}\right\},\quad j=1,\ldots, m\]
partition $\mathbb{N}.$ 
\end{theorem}

Now we jump ahead and give a proof assuming the reader is familiar with the basic concepts in symplectic geometry. Everything we use will be retrospectively introduced and accompanied with appropriate references in the sections that follow.

The proof relies on computing a certain invariant, the positive $\mathbb{S}^1$-equivariant symplectic homology - $SH^{\mathbb{S}^1,+},$ for  multidimensional ellipsoids
\begin{equation}\label{eq:ellipsoid}
E(a_1,\ldots,a_m):=\left\{(z_1,\ldots,z_m)\in\mathbb{C}^m\:\bigg|\: \frac{\abs{z_1}^2}{a_1}+\cdots+\frac{\abs{z_m}^2}{a_m}\leqslant 1\right\}
\end{equation}
in two different ways. 

\begin{proof}[Proof of Theorem~\ref{thm:main}]
Let $\lambda:=\sum_{j=1}^m\frac{x_jdy_j-y_jdx_j}{2},$
where $z_j:=x_j+iy_j$ is the $j$-th coordinate function of $\mathbb{C}^n.$ Since $\frac{a_j}{a_k}\not\in\mathbb{Q}$ for $j\not=k,$ there are exactly $m$ simple geometrically distinct periodic Reeb orbits on $\partial E(a_1,\ldots,a_m)$ with respect to the contact form induced by $\lambda.$ Those Reeb orbits are given by
\[\gamma_j:[0,\pi a_j]\to M\::\: t\mapsto \left(0,\ldots,0,e^{\frac{2i t}{a_j}},0,\ldots,0\right), \quad j=1,\ldots,m,\]
where the $j$-th coordinate is nonzero. By Theorem~\ref{thm:gutt}  and Lemma~\ref{lem:index} the $n$-th iterate $\gamma_j^n$ of $\gamma_j$ $(n\in\mathbb{N})$ contributes to $\dim SH^{\mathbb{S}^1,+}_\ast (E(a_1,\ldots,a_m))$ by 1 in degree 
\[m-1+2\sum _{k=1}^m \floor*{n\frac{a_j}{a_k}}.\]
On the other hand, Proposition~\ref{prop:chiang} implies
\begin{equation}\label{eq:SH}
\dim SH_k^{\mathbb{S}^1,+}(E(a_1,\ldots,a_m))=\left\{ \begin{matrix}1&\text{for }k=m+2j-1,\:j\in\mathbb{N},\\ 0&\text{otherwise}. \end{matrix}\right.
\end{equation}
Since $\dim SH^{\mathbb{S}^1,+}_k (E(a_1,\ldots,a_m))\leqslant 1,$ the sets $A_j,\:j=1,\ldots,m$ are pairwise disjoint. If there exists $\ell\in \mathbb{N}\setminus\bigcup_{j=1}^m A_j,$ then $\dim SH_{m-1+2\ell}^{\mathbb{S}^1,+}(E(a_1,\ldots,a_m))=0.$ This contradicts \eqref{eq:SH}. Therefore $\{A_j\}_{j=1}^m$ is a partition of $\mathbb{N}.$
\end{proof} 

\subsection*{Acknowledgments}
I would like to thank Filip Mori\'c for drawing my attention to Tamura's paper. I am also grateful to Paul Biran and Dietmar Salamon for many useful discussions.
\section{Liouville domains}

The positive $\mathbb{S}^1$-equivariant symplectic homology is an invariant associated to a certain geometric object called Liouville domain. In this section that object is defined. Along the way we introduce the basic concepts of symplectic and contact geometry. For a detailed account we refer the reader to \cite{MR1698616} and \cite{MR2397738}.

\begin{definition}
\begin{enumerate}
\item A \textbf{contact form}  is a 1-form $\alpha$ on an odd-dimensional manifold $M$ such that the top-dimensional form
\[\alpha\wedge \underbrace{d\alpha\wedge\cdots\wedge d\alpha}_{\frac{\dim M-1}{2}}\]
is nowhere vanishing.
\item A (co-oriented) \textbf{contact manifold }is an odd-dimensional manifold $M$ with a hyperplane distribution $\xi,$ called \textbf{contact distribution}, such that $\xi=\ker \alpha$ for some contact form $\alpha$ on $M.$
\item  Let $\alpha$ be a contact form on a manifold $M.$ The \textbf{Reeb vector field} $R^\alpha$ is a unique vector field on $M$ such that $\alpha(R^\alpha)=1$ and $d\alpha(R^\alpha,\cdot)=0.$ The \textbf{Reeb flow} $\sigma_t^\alpha:M\to M$ is defined by $\partial_t\sigma^\alpha_t=R^\alpha\circ\sigma_t^\alpha.$
\item A contact form $\alpha$ on $M$ is said to be \textbf{non-degenerate} if 
\[\ker\left(d\sigma_c^\alpha(x_0)-\OP{id}\right)=\left\{ t R^\alpha\:|\: t\in\mathbb{R}\right\}\]
for all $c\in(0,\infty)$ and $x_0\in M$ such that $\sigma_c^\alpha(x_0)=x_0.$
\end{enumerate}
\end{definition} 

\begin{definition}
A \textbf{symplectic form} is a closed 2-form $\omega$ on an even-dimensional manifold $W$ such that the top-dimensional form
\[\underbrace{\omega\wedge\cdots\wedge\omega}_{\frac{\dim W}{2}}\]
is nowhere vanishing. A \textbf{symplectic manifold} is a manifold with a symplectic form.
\end{definition}

\begin{definition}
A \textbf{Liouville domain} is a compact manifold (with boundary) together with a 1-form $\lambda,$ called \textbf{Liouville form}, that satisfies the following conditions. The $2$-form $d\lambda$ is a symplectic form and the \textbf{Liouville vector field} $X_\lambda,$ defined by $X_\lambda\lrcorner d\lambda=\lambda,$ points (transversally) out on the boundary.
\end{definition}

\begin{remark}
If $(W,\lambda)$ is a Liouville domain, then the restriction of $\lambda$ to the boundary $\partial W$ (i.e. the pull-back of $\lambda$ under the inclusion $\partial W\to W$) is a contact form.
\end{remark}

\begin{example}\label{ex:starshaped}
Let $W\subset \mathbb{C}^m$ be a star-shaped domain (with respect to the origin, containing the origin in the interior) with smooth boundary, and let $\lambda$ be the 1-form
\begin{equation}\label{eq:lambda}
\lambda:=\sum_{j=1}^m\frac{x_jdy_j-y_jdx_j}{2},
\end{equation}
where $z_j:=x_j+iy_j$ is the $j$-th coordinate function of $\mathbb{C}^n.$ Then, $(W,\lambda)$ is a Liouville domain. 
\end{example}

\section{Positive $\mathbb{S}^1$-equivariant symplectic homology}

The symplectic homology considered in this section is a version of $\mathbb{S}^1$-equivariant symplectic homology, that was originally defined by Viterbo \cite{MR1726235}. Rather than giving a precise definition, we will describe what kind of object $SH_\ast^{\mathbb{S}^1,+}(W,\lambda)$ is, and state its properties which are relevant to the purpose of the present paper. For more details and further references, we recommend \cite{Gutt2015}.

Throughout $(W,\lambda)$ is a Liouville domain whose first Chern class $c_1(W)$ is equal to 0. The positive $\mathbb{S}^1$-equivariant symplectic homology of $(W,\lambda)$, $SH_\ast^{\mathbb{S}^1,+}(W,\lambda),$ is a $\mathbb{Z}$-graded vector space with coefficients in $\mathbb{Q}.$

\begin{proposition}\label{prop:chiang}
Let $W\subset\mathbb{C}^{m}$ and $\lambda$ be as in Example~\ref{ex:starshaped}. Then,
\[SH_k^{\mathbb{S}^1,+}(W,\lambda)\cong\left\{ \begin{matrix}\mathbb{Q}&\text{for }k=m+2j-1,\:j\in\mathbb{N},\\ 0&\text{otherwise}. \end{matrix}\right.\]
\end{proposition}
\begin{proof}
Proposition~4.7 from \cite{MR3210581} implies the following more general statement. Let $(V,\lambda_V)$ be a so called subcritical Stein domain such that $c_1(V)$ is a torsion class ($(W,\lambda)$ belongs to this class of manifolds). Then, 
\[SH_\ast^{\mathbb{S}^1,+}(V,\lambda_V)\cong H_{\ast+d-1}(V,\partial V;\mathbb{Q})\otimes H_\ast\left(\mathbb{C}P^\infty;\mathbb{Q}\right),\]
where $2d$ is the dimension of $V.$ 
\end{proof}

To describe a convenient way of computing $SH^{\mathbb{S}^1,+}_\ast$ in certain cases, we require to introduce yet another black box - the Conley-Zehnder index $\mu_{CZ}$ for path of symplectic matrices \cite{MR1241874}. It is a half integer $\mu_{CZ}(\Psi)$ associated to the path 
\[\Psi:[a,b]\to SP(2m)\::\: t\mapsto \Psi_t.\]
Here, 
\[Sp(2n):=\left\{\Psi\in\mathbb{R}^{m\times m}\:\bigg|\:\Psi^T\left[ \begin{matrix}0&\OP{id}\\-\OP{id}& 0 \end{matrix}\right]\Psi= \left[ \begin{matrix}0&\OP{id}\\-\OP{id}& 0 \end{matrix}\right]\right\}\]
is the group of symplectic matrices. The index $\mu_{CZ}$ is equal to 0 at constant paths. It is invariant under positive reparametrizations and under homotopies with fixed end-points. Moreover it satisfies the following direct-sum relation
\begin{equation}\label{eq:directsum}
\mu_{CZ}(\Phi\oplus\Psi)= \mu_{CZ}(\Phi)+\mu_{CZ}(\Psi),
\end{equation}
where $\Phi:[a,b]\to Sp(2m),\:\Psi:[a,b]\to Sp(2m'),$ and
\[\Phi\oplus\Psi:[a,b]\to Sp(2m+2m')\::\: t\mapsto \left[\begin{matrix} \Phi_t& 0\\ 0&\Psi_t\end{matrix}\right].\]

Now we explain how to compute $\mu_{CZ}$ for generic smooth $\Psi:[a,b]\to Sp(2m).$

\begin{definition}
Let $\Psi:[a,b]\to Sp(2m)$ be a smooth path. A number $t\in[a,b]$ is called \textbf{crossing} if $\det(\Psi_t-\OP{id})=0.$ The \textbf{crossing form} at the crossing $t$ is the quadratic form
\[\Gamma(\Psi,t):\ker(\Psi_t-\OP{id})\times \ker(\Psi_t-\OP{id})\to \mathbb{R}\::\: (\zeta,\eta)\mapsto \zeta^TS_t\eta,\]
where $S_t$ is a symmetric matrix defined by
\[S_t:=\left[ \begin{matrix}0&\OP{id}\\-\OP{id}& 0 \end{matrix}\right] (\partial_t\Psi_t)\Psi_t^{-1}.\]
The crossing is said to be \textbf{non-degenerate} if the crossing form is non-degenerate.
\end{definition}

\begin{proposition}\label{prop:mczformula}
Let $\Psi:[a,b]\to Sp(2m)$ be a smooth path such that all the crossings are isolated and non-degenerate.Then,
\begin{equation}\label{eq:mczformula}
\mu_{CZ}(\Psi)=\frac{1}{2}\OP{sign}\Gamma(\Psi,a)+\sum_{\begin{matrix}t\in(a,b)\\ \text{crossing}\end{matrix}} \OP{sign}\Gamma(\Psi,t)+ \frac{1}{2}\OP{sign}\Gamma(\Psi,b).
\end{equation}
Here, $\OP{sign}$ denotes the signature (i.e. the number of positive minus the number of negative eigenvalues). If $a$ is not a crossing, then $\OP{sign}\Gamma(\Psi,a)$ is defined to be 0. Analogously for $b.$
\end{proposition}
\begin{proof}
See \cite{MR1241874}. Compere also with \cite[Section~3]{MR3210581}.
\end{proof}

\begin{example}\label{ex:rotations}
Let $\alpha\in(0,\infty).$ Consider the path of symplectic matrices
\[\Psi:[0,2\pi T]\to Sp(2)\::\: t\mapsto \left[\begin{matrix}\cos(\alpha t)&-\sin(\alpha t)\\ \sin(\alpha t)& \cos(\alpha t) \end{matrix}\right].\]
A number $c\in[0, 2\pi T]$ is a crossing if and only if $c=\frac{2k\pi}{\alpha}$ for some $k\in\mathbb{Z}.$ In that case $\Psi_c=\OP{id}.$ Hence $\ker(\Psi_c-\OP{id})=\mathbb{R}^2.$ The crossing form is equal to
\[\Gamma(\Psi,c):\mathbb{R}^2\times\mathbb{R}^2\to \mathbb{R}\::\:(\zeta,\eta)\mapsto \zeta^T\left[\begin{matrix}\alpha&0\\0&\alpha \end{matrix}\right]\eta\]
for all crossings.  Therefore the crossings are non-degenerate and $\OP{sign} \Gamma(\Psi,c)=2.$ By Proposition~\ref{prop:mczformula}, we get
\[\mu_{CZ}=\left\{ \begin{matrix} 1+2\floor*{T\alpha} &\text{if } T\alpha\not\in \mathbb{Z}\\ T\alpha &\text{if } T\alpha\in \mathbb{Z}.\end{matrix} \right.\]
\end{example}

\begin{definition}\label{def:indexcontact}
Let $(M,\xi=\ker\alpha)$ be a contact manifold.
\begin{enumerate}
\item Let $x_0\in M$ be such that the map 
\[\mathbb{R}\to M\::\: t\mapsto \sigma_t^\alpha(x_0)\]
is $c$-periodic, $c\in(0,\infty).$ The map
\[[0,c]\to M\::\:t\mapsto \sigma_t^\alpha(x_0)\]
is called \textbf{($c$-)periodic Reeb orbit}. It is said to be \textbf{contractible} if it is contractible seen as a loop $\mathbb{R}/c\mathbb{Z}\to M.$
\item Let $\gamma:[0,c]\to M$ be a contractible periodic Reeb orbit, and let 
\[u:\{z\in\mathbb{C}\:|\:\abs{z}\leqslant 1\}=:\mathbb{D}\to M\]
be a \textbf{capping disk} of $\gamma,$ i.e. it satisfies $u\left(e^{\frac{2\pi}{c}t}\right)=\gamma(t).$ A trivialization $T_t:\mathbb{R}^{\dim M-1}\to \xi_{\gamma(t)}$ of the bundle $\gamma^\ast \xi$ is said to be a \textbf{capping trivialization} induced by the capping disk $u$ if
\[T_t=T'_{e^{2\pi t/c}},\]
where $T'_z:\mathbb{R}^{\dim M-1}\to \xi_{u(z)}$ is a trivialization of the bundle $u^\ast \xi$ such that
\[d\alpha(T'\zeta,T'\eta)=\left[\begin{matrix} 0& \OP{id}\\ -\OP{id}& 0\end{matrix}\right].\]
\item Let $\gamma:[0,c]\to M$ be a contractible periodic Reeb orbit and let $u:\mathbb{D}\to M$ be a capping disk of $\gamma.$ Choose a capping trivialization $T_t:\mathbb{R}^{\dim M-1}\to \xi_{\gamma(t)}$ induced by the capping disk $u.$ The \textbf{Conley-Zehnder index} of the Reeb orbit $\gamma$ with respect to the capping disk $u$ is definded to be the Conley-Zehnder index of the path of symplectic matrices
\[[0,c]\mapsto Sp(\dim M-1)\::\: t\mapsto T^{-1}_t\circ d\sigma^\alpha_t(\gamma(0))\circ T_0.\]
It is denoted by $\mu_{CZ,u}(\gamma).$ As it is suggested by the notation, the Conley-Zehnder index does not depend on the choice of capping trivialization induced by $u.$ Moreover, if $\pi_2(M)=0$ it does not depend on the choice of capping disk either. In that case it is denoted simply by $\mu_{CZ}(\gamma).$
\end{enumerate}
\end{definition}

\begin{definition}
Let $(W,\omega)$ be a symplectic manifold.
\begin{enumerate}
\item A \textbf{Hamiltonian} on $W$ is a smooth function $H:W\to \mathbb{R}.$ To each Hamiltonian $H,$ one can associate the \textbf{Hamiltonian vector field} $X_H$ on $M,$ defined by $\omega(X_H,\cdot)=dH,$ and the \textbf{Hamiltonian flow} $\psi^H_t:W\to W$ furnished by $X_H,$ i.e. $\partial_t \psi_t^H=X_H\circ\psi^H_t.$  
\item Let $\gamma:[0,c]\to W$ be a $c$-periodic Hamiltonian orbit, i.e. $\gamma(t)=\psi^H_t(x_0)$ where $x_0\in W$ is such that the map $\mathbb{R}\to W\::\:t\mapsto \psi^H_t(x_0)$ is $c$-periodic. A capping disk $u:\mathbb{D}\to W$ and a capping trivialization $T_t:\mathbb{\dim W}\to T_{\gamma(t)W}$ are defined as in Definition~\ref{def:indexcontact} (with the contact distribution $\xi$ replaced by $TW$). The \textbf{Conley-Zehnder index} of $\gamma$ with respect to a capping trivialization $T$ induced by a capping disk $u:\mathbb{D}\to W,$ $\overline{\mu}_{CZ}(\gamma),$ is defined  to be equal to the Conley-Zehnder index of the path of symplectic matrices
\[[0,c]\mapsto SP(\dim W)\::\: t\mapsto T^{-1}_t\circ d\psi^H_t(\gamma(0))\circ T_0.\]
It does not depend on the choice of capping trivialization induced by $u:\mathbb{D}\to W,$ and, if $\pi_2(W)=0,$ it is also independent of the choice of capping disk.
\end{enumerate}
\end{definition}

\begin{definition}
Let $M$ be a manifold with a contact form $\alpha.$ Assume $\gamma:[0,c]\to M$ is a simple periodic Reeb orbit (i.e. the restriction of $\gamma$ to $(0,c)$ is injective). For $k\in\mathbb{N},$ denote by $\gamma^k$ the $k$-fold iterate of $\gamma$
\[\gamma^k:[0,kc]\to M\::\: t\mapsto \gamma(t).\]
The orbit $\gamma^k$ is said to be \textbf{good} if 
\[\mu_{CZ}(\gamma^k)\equiv \mu_{CZ}(\gamma)\pmod{2}. \] 
\end{definition}

\begin{definition}
A subset $A\subset\mathbb{Z}$ is \textbf{lacunary} if it does not contain a pair of consecutive integers.
\end{definition}

\begin{theorem}[Gutt]\label{thm:gutt}
Let $(W,\lambda)$ be a Liouville domain with $c_1(W)=0$ and simply-connected boundary.\footnote{The condition $\pi_1(\partial W)=0$ is not essential. We imposed it here to make our exposition on $\mu_{CZ}$ somewhat simpler.} Assume there exists a non-degenerate contact form $\alpha$ on the boundary $\partial W$ such that the set of Conley-Zehnder indices of all good periodic Reeb orbits is lacunary. Then
\[SH^{\mathbb{S}^1,+}_k(W,\lambda)=\bigoplus_{\gamma}\mathbb{Q}\left\langle\gamma\right\rangle,\]
where the sum goes through the set of good geometrically distinct\footnote{If $\gamma$ is a periodic Reeb orbit, then $\gamma(\cdot+a)$ is as well. However, these Reeb orbits are considered to be geometrically same.} periodic ($\alpha$-)Reeb orbits on $\partial W$ with $\mu_{CZ}=k.$
\end{theorem}
\begin{proof}
See \cite[Theorem~1.1 and Corollary~3.7]{Gutt2015}.
\end{proof}

\section{The case of irrational ellipsoids}

Throughout this section $m\geqslant 2$ is an integer and $a_1,\ldots, a_m\in (0.\infty)$ are real numbers such that $\frac{a_j}{a_k}\not\in\mathbb{Q}$ for $j\not=k.$ We consider the Liouville domain $(E(a_1,\ldots,a_m),\lambda)=:(W,\lambda)$ where $E(a_1,\ldots,a_m)\subset\mathbb{C}^m$ is defined by \eqref{eq:ellipsoid} and $\lambda$ by \eqref{eq:lambda}. The coordinate functions on $\mathbb{C}^m$ are denoted by $z_j=x_j+iy_j,\:j=1,\ldots, m.$

The Reeb vector field with respect to the contact form induced by $\lambda$ on the boundary is given by
\[R= \sum_{j=1}^m\frac{2}{a_j}\left( -y_j\partial_{x_j}+x_j\partial_{y_j} \right).\]
it generates the Reeb flow
\[\sigma_t:\mathbb{C}^m\to\mathbb{C}^m\::\: (z_1,\ldots, z_m)\mapsto \left( e^{\frac{2ti}{a_1}}z_1,\ldots,  e^{\frac{2ti}{a_m}}z_m \right).\]
Since $\frac{a_j}{a_k}\not\in\mathbb{Q},$ for $j\not=k,$ there are exactly $m$ simple geometrically distinct periodic Reeb orbits on $\partial W$ 
\[\gamma_j:[0,\pi a_j]\to M\::\: t\mapsto \sigma_t(pj)=\left(0,\ldots,0,e^{\frac{2i t}{a_j}},0,\ldots,0\right), \quad j=1,\ldots,m.\]
Here $p_j:=(0,\ldots,0,1,0,\ldots,0)\in\mathbb{C}^m,$ where $1$ is on the $j$-th place.

\begin{lemma}\label{lem:index}
Let $\gamma_j,\:j\in\{1,\ldots,m\}$ be as above and let $n\in\mathbb{N}.$ Then, the Conley-Zehnder index of the $n$-th iterate $\gamma^n_j$ of $\gamma_j$ is equal to
\[\mu_{CZ}(\gamma^n_j)=m-1+2\sum _{k=1}^m \floor*{n\frac{a_j}{a_k}}.\]
\end{lemma}
\begin{proof}
Since it is quite challenging to find explicitly a capping trivialization of the contact distribution $\xi$ on $\partial W$ along $\gamma_j^n,$ we will compute $\mu_{CZ}(\gamma^n_j)$ indirectly. Namely, we will find a Hamiltonian $H:W\to\mathbb{R}$ such that $R=X_H$ on $\partial W$ (and, consequently, the Hamiltonian flow $\psi_t^H$ agrees with the Reeb flow $\sigma_t$ on $\partial W$). For such a Hamiltonian $\gamma^n_j$ is a periodic Hamiltonian orbit. Hence $\overline{\mu}_{CZ}(\gamma_j^n)$ is well defined and not difficult to compute since $TW=W\times\mathbb{R}^{2m}.$ We will finish the proof by finding a relation between $\overline{\mu}_{CZ}(\gamma_j^n)$ and $\mu_{CZ}(\gamma^n_j).$

The Hamiltonian 
\[H(z_1,\ldots, z_m):=-\left(\frac{\abs{z_1}^2}{a_1}+\cdots+\frac{\abs{z_m}^2}{a_m}\right)\]
satisfies the conditions stated above. Its vector field and flow are given by
\[X_H=\sum_{j=1}^m\frac{2}{a_j}\left( -y_j\partial_{x_j}+x_j\partial_{y_j} \right),\quad \psi^H_t(z_1,\ldots, z_m)= \left( e^{\frac{2ti}{a_1}}z_1,\ldots,  e^{\frac{2ti}{a_m}}z_m \right).\]
The Conley-Zehnder index $\overline{\mu}_{CZ}(\gamma_j^n)$ is equal to the Conley-Zehnder index of the following path of symplectic matrices
\[[0,n\pi a_j]\to SP(2m)\::\: t\mapsto d\psi^H_t(q_j)=\bigoplus_{l=1}^m\left[ \begin{matrix} \cos\left(\frac{2it}{a_l}\right)&-\sin\left(\frac{2it}{a_l}\right)\\ \sin\left(\frac{2it}{a_l}\right) & \cos\left(\frac{2it}{a_l}\right) \end{matrix}\right].\]
By \eqref{eq:directsum} and Example~\ref{ex:rotations}, we get
\[\overline{\mu}_{CZ}(\gamma_j^n)=\sum_{l=1,l\not=j}^m\left(1+2\floor*{\frac{2n\pi a_j}{2\pi a_j}}\right)+ 2\frac{2n\pi a_j}{2\pi a_j}= m-1+ 2\sum_{l=1}^m\floor*{n\frac{a_j}{a_l}}.\]
On the other hand, a capping trivialization $T_t:\mathbb{R}^{2m-2}\to\xi_{\gamma^n_j(t)},$ used to compute $\mu_{CZ}(\gamma_j^n),$ gives rise to the capping trivialization $T'_t:\mathbb{R}^{2m}\to T_{\gamma_j^n(t)}W$ defined by
\[T'_t(v_1,w_1,\ldots,v_m,w_m)=v_1 X_\lambda+w_1 R+ T_t(v_2,w_2,\ldots, v_m,w_m).\]
Using this trivialization, and noticing that $\psi^H_t$ preserves $X_\lambda$ and $R,$ we conclude that $\overline{\mu}_{CZ}(\gamma_j^n)$ is equal to the Conley-Zehnder index of the following path of symplectic matrices
\[\Phi:[0,n\pi a_j]\to Sp(2m)\::\:t\mapsto \begin{bmatrix}1&0\\0&1\end{bmatrix}\oplus \Psi_t, \]
where $\Psi_t:= T_t^{-1}\circ d\sigma_t(q_j)\circ T_0.$
The Conley -Zehnder index of the constant path of symplectic matrices is equal to 0. This together with \eqref{eq:directsum} implies 
\[\overline{\mu}_{CZ}(\gamma_j^n)=\mu_{CZ}(\Phi)=\mu_{CZ}(\Psi).\]
And by definition $\mu_{CZ(\gamma^n_j)}=\mu_{CZ}(\Psi).$ Hence $\overline{\mu}_{CZ(\gamma^n_j)}=\mu_{CZ(\gamma^n_j)}$ and the proof is finished. 
\end{proof}

\printbibliography
\end{document}